\documentclass{amsart}
\pdfoutput=1

\usepackage[utf8]{inputenc}
\usepackage[T1]{fontenc}
\usepackage{lmodern}
\usepackage[margin=1.25in]{geometry}
\usepackage{setspace}
\usepackage[usenames, dvipsnames]{xcolor}
\usepackage{graphicx}
\usepackage{mathtools}
\usepackage{amssymb}
\usepackage{amsthm}
\usepackage{xspace}
\usepackage{tikz-cd}
\usepackage{tkz-euclide}
\usepackage{tkz-graph}
\usetikzlibrary{
  knots,
  hobby,
  decorations.pathreplacing,
  shapes.geometric,
  calc
}

\usepackage{microtype}       %
\usepackage{bbm}
\usepackage{slashed}
\usepackage[backref=page, bookmarks=false]{hyperref}
\usepackage[capitalize, noabbrev]{cleveref}
\newcommand{\Z}{\mathbb Z}

\newsavebox{\pullback}
\sbox\pullback{%
\begin{tikzpicture}%
\draw (0,0) -- (1ex,0ex);%
\draw (1ex,0ex) -- (1ex,1ex);%
\end{tikzpicture}}

\newtheorem{thm}{Theorem}[section]
\newtheorem{lemma}[thm]{Lemma}

\newtheorem{theorem}[thm]{Theorem}

\theoremstyle{definition}
\newtheorem{example}[thm]{Example}
\newtheorem{definition}[thm]{Definition}

\newtheorem{introthm}{Theorem}

\theoremstyle{remark}
\newtheorem{remark}[thm]{Remark}

\DeclarePairedDelimiter\paren{(}{)}

\makeatletter
	\let\oldparen\paren
	\def\paren{\@ifstar{\oldparen}{\oldparen*}}
\makeatother

\usepackage{microtype}
\usepackage{hypcap}

\hypersetup{
 colorlinks,
 linkcolor={red!50!black},
 citecolor={green!50!black},
 urlcolor={blue!80!black}
}

 \usepackage{textcomp}
 
\newcommand{\Zbb}{\mathbb{Z}}
\newcommand{\Qbb}{\mathbb{Q}}

\usepackage{adjustbox}
\usepackage{caption}

\tikzset{%
    symbol/.style={%
        draw=none,
        every to/.append style={%
            edge node={node [sloped, allow upside down, auto=false]{$#1$}}}
    }
}

\makeatletter
\newcommand{\colim@}[2]{%
  \vtop{\m@th\ialign{##\cr
    \hfil$#1\operator@font colim$\hfil\cr
    \noalign{\nointerlineskip\kern-\ex@}\cr}}%
}

\newcommand{\colim}{%
  \mathop{\mathpalette\colim@{\scriptscriptstyle}}\nmlimits@
}

\begin{document}

\title{Tensor Product $K$-theory is Rational Algebraic $K$-theory}

\author{Amartya Shekhar Dubey}
\address{National Institute of Science Education and Research, Homi Bhabha National Institute  }
\curraddr{}
\email{amartyashekhar.dubey@niser.ac.in}
\thanks{}

\author{Mattie Ji}
\address{University of Pennsylvania}
\curraddr{}
\email{mji13@sas.upenn.edu}
\thanks{}


\begin{abstract}
For a commutative ring $R$ with unity, its algebraic $K$-theory space $K(R)$ may be obtained by group-completing the symmetric monoidal category of finitely generated free $R$-modules under direct sum. A natural question is what happens when one group-completes with respect to the tensor product structure instead. In this note, we give a direct proof of the folklore theorem that the resulting group-completion is the rationalization of $K(R)$, up to $\pi_0$. We also discuss how a similar group-completion would give the $p$-perfection and, more generally, the localization of $K(R)$ at any non-trivial multiplicatively closed subset $S \subseteq \mathbb{Z}_{> 0}$. The localization statement can be recovered from a localization theorem of May. We give a plus-construction proof without using the full machinery of multiplicative infinite loop space theory.
\end{abstract}

\maketitle

\begin{figure*}[!htbp]
\centering
\includegraphics[width=0.75\textwidth]{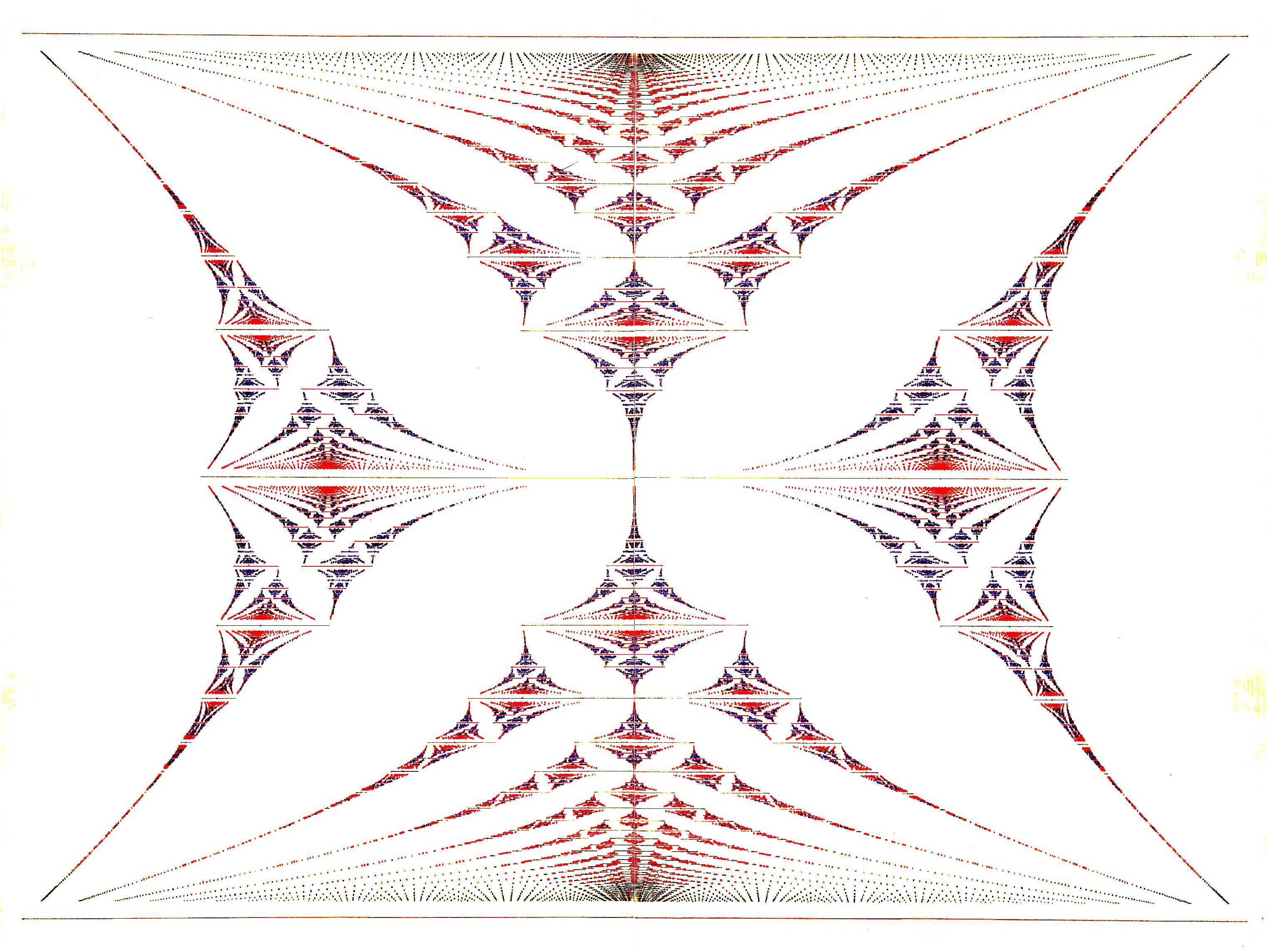}
\caption*{The Hofstadter butterfly: a fractal spectrum (in the sense of functional analysis) whose every gap is labeled by a $K$-theory class. \textit{A butterfly that only exists because $K$-theory closed its wings!} Image by Douglas Hofstadter (1976), sourced from Wikimedia Commons \cite{Hofstadter1976Gplot}, public domain.}
\end{figure*}
\clearpage

\section{Algebraic $K$-Theory, Tensor Product $K$-Theory, and History}

For a commutative ring $R$ with unity, its algebraic $K$-theory $K(R)$ has been increasingly relevant in a variety of areas in mathematics (see \cite{weibel2013k} for a comprehensive coverage). One way to define this is via Segal's $K$-theory of symmetric monoidal categories \cite{Segal1974}.

\begin{definition}
    For a small symmetric monoidal category $\mathcal{C}$, we write $K(\mathcal{C})$ to denote Segal's $K$-theory space of $\mathcal{C}$. We write $\tau_{>0} K(\mathcal{C})$ to denote the component of $K(\mathcal{C})$ at the identity. We write $K_i(\mathcal{C})$ to denote $\pi_i(K(\mathcal{C}))$. $K(\mathcal{C})$ has a canonical structure of being an infinite loop space (see \cite{may_operads, Segal1974, may_thomason_uniqueness}).
\end{definition}

\subsection{Tensor Product $K$-theory and Rationalization} When $\mathcal{C} = \operatorname{Free}_{\oplus}(R)$ is the symmetric monoidal category of finitely generated free $R$-modules (e.g., $R^n$ for $n \geq 0$), with morphisms being $R$-linear automorphisms, equipped with the symmetric monoidal structure of direct sums, the classical $K$-theory space $K(R)$ is given by
    \[K(R) \coloneqq K_0(R) \times \tau_{>0} K(\operatorname{Free}_{\oplus}(R)),\]
where $K_0(R)$ has the discrete topology. From here we also see that $\tau_{>0} K(R) = \tau_{>0} K(\operatorname{Free}_{\oplus}(R))$. Algebraic $K$-theory can then be morally regarded as a \textit{group completion} with respect to the \textit{direct sum structure} of $R$-modules. One natural question one might wonder is: what if we \textit{group complete} with respect to the \textit{tensor product structure} instead?

\begin{definition}
    We use $\operatorname{Free}_{\otimes}(R)$ to denote the symmetric monoidal category of non-zero finitely generated free $R$-modules (e.g., $R^n$ for $n > 0$), with morphisms being the automorphism groups $\operatorname{GL}_{n}(R)$ of each $R^n$, equipped with the symmetric monoidal structure of tensor products.
\end{definition}

For $i > 0$, $K_i(\operatorname{Free}_{\otimes}(R))$ may be thought of as a tensor version of $K$-theory, or \textit{tensor product $K$-theory}. The purpose of this write-up is to prove the following folklore result.
\begin{introthm}\label{thm::k_theory_rationalization}
The space $\tau_{> 0} K(\operatorname{Free}_{\otimes}(R))$ is equivalent to the rationalization of $\tau_{> 0} K(R)$. Furthermore, they are equivalent as infinite loop spaces.
\end{introthm}

In Section~\ref{sec::plus}, we will define more formally what rationalization (and certain localizations) of infinite loop spaces are. Here we explain some more contexts and motivations behind Theorem~\ref{thm::k_theory_rationalization}.

\begin{remark}
Theorem~\ref{thm::k_theory_rationalization} can be deduced from Proposition 5 of \cite{c80f3555-900f-35b5-ab05-6d397e6a44e6}, which states that $K_i(\operatorname{FP}) = K_i(R) \otimes \mathbb{Q}$ for $i \geq 1$ for any commutative ring $R$, where $\operatorname{FP}$ denotes the category of faithfully projective $R$-modules with the symmetric monoidal operation of tensor product. This fact was also pointed out in \cite{951cf383-5a4d-3775-b8b2-0bafd2f162fe}. Weibel's papers do not include a proof, but they both indicate that it follows from May's book \textit{$E_{\infty}$ Ring Spaces and $E_{\infty}$ Ring Spectra} \cite{May1977EInfinityRingSpaces}, Theorem VII.5.3 therein (see also Theorem 10.2 of \cite{may2009preciselyeinftyringspaces}). Theorem VII.5.3 of \cite{May1977EInfinityRingSpaces} requires a certain homological stability condition called being \textit{convergent}. Furthermore, the theorem requires some technical corrections, as explained in Remarks 8.7 of a follow-up work \cite{MAY19821}, which also mentions the assumption of a homological stability condition.\\

\noindent May’s result, while more general, requires technical machinery from operads and homology of $E_{\infty}$-ring spaces and assumes a homological stability condition. Our methods for rings with homological stability does not use the full machinery of multiplicative infinite loop space theory used in May's proof of the localization theorem. We also write down an explicit way to generalize this to rings that do not satisfy the appropriate homological stability conditions. To the best of our knowledge, an explicit proof for Theorem~\ref{thm::k_theory_rationalization} had never been written down, and the statement does not appear to be well-known in the $K$-theory community. We hope it would be beneficial to provide a proof here.
\end{remark}

\begin{remark}
    In Theorem IX.7.3 of \cite{bass1968algebraic}, Bass proved $K_1(\operatorname{Free}_{\otimes}(R)) = K_1(R) \otimes \Qbb$ without any assumption on $R$. Bass's proof relied on the crucial fact that in sufficiently large dimensions, the maps $f: A \mapsto \operatorname{diag}(A^k, 1, ..., 1)$ and $g: A \mapsto \operatorname{diag}(A, ..., A) =: I_{k} \otimes A$ are the same modulo commutators. This implies they induce the same maps on $H_1(\operatorname{BGL}_n(R)) \to H_1(\operatorname{BGL}_{kn}(R))$. However, it seems unlikely that $f$ and $g$ would induce the same map in higher homology groups, as the maps $Bf, Bg: \operatorname{BGL}_{n}(R) \to \operatorname{BGL}_{kn}(R)$ are not freely homotopic.
\end{remark}

\begin{remark}
In the recent paper \cite{ji2026quantumcellularautomatagroup}, Ji and Yang construct a multiplicative version of negative $K$-theory that deloops the $K$-theory of the symmetric monoidal category $\operatorname{Az}(R)$ of Azumaya $R$-algebras. Furthermore, they show that this delooping is intricately related to the theory of (algebraic) \textit{quantum cellular automata} (QCA, see  \cite{farrelly2020review}, Section 2 of \cite{ji2026quantumcellularautomatagroup}, or Appendix A of \cite{Yang_2026}). The computation for the higher homotopy groups of $K(\operatorname{Az}(R))$ involves the use of  Theorem~\ref{thm::k_theory_rationalization}. The method to compute $K_i(\operatorname{Az}(R))$ using Theorem~\ref{thm::k_theory_rationalization} was originally carried out in \cite{c80f3555-900f-35b5-ab05-6d397e6a44e6} and partially explained in the expository Section 6 of \cite{ji2026quantumcellularautomatagroup}, which also included an exposition on why $K_1(\operatorname{Free}_{\otimes}(R)) = K_1(R) \otimes \Qbb$ for an Euclidean domain $R$. In light of the theme of QCA in \cite{ji2026quantumcellularautomatagroup}, Theorem~\ref{thm::k_theory_rationalization} may also be of independent interest in quantum physics.
\end{remark}

\subsection{$K$-theory and $p$-Perfection} The strategy here to prove Theorem~\ref{thm::k_theory_rationalization} is also quite flexible and generalizable. As a proof of concept, we will first prove the $p$-perfection (also known as localization away from $p$, or the $p$-telescope) version of Theorem~\ref{thm::k_theory_rationalization} and then see how to prove Theorem~\ref{thm::k_theory_rationalization} analogously. To be precise, we consider the following category.

\begin{definition}
Let $p$ be a prime number. We use $\operatorname{Free}_{\otimes}^p(R)$ to denote the full subcategory of $\operatorname{Free}_{\otimes}(R)$ on $R^1, R^p, R^{p^2}, R^{p^3}, ...$. This is a symmetric monoidal subcategory of $\operatorname{Free}_{\otimes}(R)$.
\end{definition}

\begin{introthm}\label{thm::k_theory_p_perfection}
The space $\tau_{>0} K(\operatorname{Free}_{\otimes}^p(R))$ is equivalent to $(\tau_{>0} K(R))[\frac{1}{p}]$.
\end{introthm}

This statement also appears in \cite{ramzi2025pperfectiongroupcompletionmathbbeinftymonoids} as a consequence of their more general theorem for $p$-perfection of commutative monoids.

\subsection{Generalizations to Multiplicative Localizations} A subset $S \subseteq \Zbb_{> 0}$ is (non-trivially) \textit{multiplicatively closed} if $\{1\} \subsetneq S$ and for any $a, b \in S$, their product $ab$ is in $S$. For any multiplicatively closed subset $S \subseteq \Zbb_{>0}$, we can obtain a localization of $\tau_{>0} K(R)$ at $S$ denoted $(\tau_{>0} K(R))[S^{-1}]$. The rationalization and $p$-perfect of $\tau_{>0} K(R)$ are special cases of localization at the set $S = \Zbb_{>0}$ and $S = \{1, p^1, p^2, p^3, ...\}$ respectively. Theorem~\ref{thm::k_theory_rationalization} and~\ref{thm::k_theory_p_perfection} can be generalized to the following statement about multiplicative localizations.

\begin{definition}\label{def::s_cat}
    Let $S \subseteq \Zbb_{>0}$ be a multiplicatively closed subset.  $\operatorname{Free}^{S}_{\otimes}(R)$ is the full subcategory of $\operatorname{Free}_{\otimes}(R)$ consisting of $R^s$ for all $s \in S$. This is a symmetric monoidal subcategory of $\operatorname{Free}_{\otimes}(R)$.
\end{definition}

\begin{introthm}\label{thm::k_theory_localization}
The space $\tau_{>0} K(\operatorname{Free}_{\otimes}^S(R))$ is equivalent to $(\tau_{>0} K(R))[S^{-1}]$.
\end{introthm}

Although Theorem~\ref{thm::k_theory_rationalization} and~\ref{thm::k_theory_p_perfection} will directly follow from Theorem~\ref{thm::k_theory_localization}, for the ease of presentation and understanding, we will prove those two special cases first to motivate how to generalize. Theorem~\ref{thm::k_theory_localization} can also be obtained from Theorem VII.5.3 in \cite{May1977EInfinityRingSpaces} under a homological stability assumption on $R$. We again provide an explicit proof of this for all rings. To the best of our knowledge, a proof for Theorem~\ref{thm::k_theory_localization} had never been explicitly written down.

\subsection{Relations to the Plus Construction}\label{sec::plus_relation} Although we motivated the topic of this paper using Segal's $K$-theory, we could have also done this in the language of Quillen's plus construction. In fact, our proofs of Theorem~\ref{thm::k_theory_rationalization}, ~\ref{thm::k_theory_p_perfection}, and~\ref{thm::k_theory_localization} are done in the language of plus constructions.

For a path-connected space $X$ and its maximal perfect normal subgroup $N \subseteq \pi_1(X)$, Quillen's plus construction is a pair $(X^+, i: X \to X^+)$ such that (a) $i$ is an isomorphism on homology in any local coefficient systems on $X^+$ and (b) $N = \ker(i_*: \pi_1(X) \to \pi_1(X^+))$. The map $i$ is initial among all maps $f: X \to Y$ such that $N \subseteq \ker(f_*: \pi_1(X) \to \pi_1(Y))$. For the usual $K$-theory space $K(R)$, \cite{quillen_2} showed that $\tau_{>0} K(R) \simeq \operatorname{BGL}(R)^+$,
where $\operatorname{GL}(R) = \operatorname{colim}_{n \geq 1} \operatorname{GL}_n(R)$ with transition maps $i: \operatorname{GL}_n(R) \to \operatorname{GL}_{n+1}(R)$ being the inclusion of a matrix $A$ to the block diagonal matrix $\begin{pmatrix}
    A & 0\\
    0 & 1
\end{pmatrix}$.

The spaces $\tau_{>0} K(\operatorname{Free}^p_{\otimes}(R))$, $\tau_{>0} K(\operatorname{Free}_{\otimes}(R))$ and $\tau_{>0} K(\operatorname{Free}_{\otimes}^S(R))$ also admit similar plus construction descriptions (see Lemma~\ref{lem::plus_construction_des}). Thus, Theorem~\ref{thm::k_theory_rationalization},~\ref{thm::k_theory_p_perfection}, and~\ref{thm::k_theory_localization} can also be seen as statements about the plus-construction versions of tensor product $K$-theory to the usual algebraic $K$-theory.\\

\noindent \textbf{Outline.} In Section~\ref{sec::plus}, we will define the localization of an infinite loop space at $S$, recast Theorem~\ref{thm::k_theory_rationalization},~\ref{thm::k_theory_p_perfection}, and~\ref{thm::k_theory_localization} to statements about plus-construction spaces and explain a homological stability condition on $R$ that will be used later in proofs. In Section~\ref{sec::p_perfection}, we will prove Theorem~\ref{thm::k_theory_p_perfection}. In Section~\ref{sec::rationalization}, we will prove Theorem~\ref{thm::k_theory_rationalization}. In Section~\ref{sec::k_theory_localization}, we will generalize the methods used to prove Theorem~\ref{thm::k_theory_localization}.\\

\noindent \textbf{Acknowledgements.} MJ would like to thank Mark Behrens, Shane Kelly, Peter May, Mona Merling, Charles Weibel, and Bowen Yang for helpful conversations, and again to Bowen Yang for the question that motivated this project. ASD would like to thank Lennart Meier and Shachar Carmeli for their support and encouragement. Both authors thank Maxime Ramzi for helpful conversations. MJ is partially supported by the National Science Foundation Graduate Research Fellowship (DGE-2236662). ASD acknowledges the support of the National Institute of Science Education and Research (NISER) and Homi Bhabha National Institute (HBNI), Mumbai. \\

\noindent \textbf{Notations.} Throughout for a group $G$, the classifying space $BG$ is the Eilenberg-MacLane space $K(G, 1)$, which is characterized by the property that $\pi_1(BG) = G$ and $\pi_k(BG) = 0$ for $k \neq 1$. Unless mentioned otherwise, $H_*$ denotes integral homology.

\section{Preliminaries: Localizations, Plus-Constructions, and Homological Stability}\label{sec::plus}

\subsection{Localizations} 

Let $S \subseteq \mathbb{Z}_{> 0}$ be a multiplicatively closed subset and $X$ be an infinite loop space. Here we make precise what the infinite loop space $X[S^{-1}]$ means. Note that this construction can be reasonably adapted to an $H$-space (see Proposition 6.6.1 of \cite{may2011more}).

The infinite loop space $X$ is equipped with a binary operation $\mu: X \times X \to X$.
\begin{definition}
    Let $n \in \Zbb_{>0}$, the map $\cdot n: X \to X$ is the $n$-th fold power given as the composition:
    \[X \xrightarrow{\Delta} X^n \xrightarrow{\mu_n} X,\]
    where $\Delta$ is the diagonal map and $\mu_n: X^n \to X$ is the $n$-fold evaluation map with $\mu$ by associating the products in a fixed order.
\end{definition}


\begin{definition}\label{def::localization}
Write the elements of $S$ in increasing order as $1 = a_1 < a_2 < a_3 < a_4 < ... < ...$ The localization of $X$ at $S$ is
\[X[S^{-1}] \coloneqq \operatorname{colim}(X \xrightarrow{\cdot a_2} X \xrightarrow{\cdot a_3} X \xrightarrow{\cdot a_4} ...).\]
In special cases, we write the rationalization and $p$-perfection of $X$ as 
$$X \otimes \mathbb{Q} \coloneqq X[\Zbb_{>0}^{-1}] \text{ and } X[p^{-1}] \coloneqq X[\{1, p, p^2, ...\}^{-1}] \cong \operatorname{colim}(X \xrightarrow{\cdot p} X \xrightarrow{\cdot p} X \xrightarrow{\cdot p }...).$$
\end{definition}

\begin{remark}
Since $X$ is an infinite loop space, the multiplication maps $\cdot n: X \to X$ for $n \in \mathbb{Z}_{>0}$ form coherently multiplicative system, being built from the $E_\infty$-structure on $X$. Thus they determine a homotopy coherent diagram indexed by the filtered poset $S$ (under divisibility), and the homotopy colimit of infinite loop space
\[
X[S^{-1}] \coloneqq \operatorname{hocolim}_{s \in S} X
\]
is well-defined. The underlying space is the usual homotopy colimit of spaces, as the forgetful functor from infinite loop space to spaces preserve filtered colimits (see the proof of Proposition B.8 of \cite{Gepner_2015}). The sequential description in Definition~\ref{def::localization} is obtained by restricting this homotopy colimit to a cofinal chain in $S$.
\end{remark}

\begin{example}
The case of particular interest to us is when $X = \tau_{>0} K(R)$. The definitions above allow us to make sense of $\tau_{>0} K(R) \otimes \mathbb{Q}, (\tau_{>0}K(R))[p^{-1}]$, and $(\tau_{>0} K(R))[S^{-1}]$.
\end{example}

\subsection{Plus-Constructions}
Here we interpret the spaces $\tau_{>0} K(\operatorname{Free}_{\otimes}(R))$, $\tau_{>0} K(\operatorname{Free}_{\otimes}^p(R))$, and $\tau_{>0} K(\operatorname{Free}^{S}_{\otimes}(R))$ as plus construction spaces.

\begin{definition}
For natural numbers $n = kd$, we write the map
\[\phi_{d, n}: \operatorname{GL}_{d}(R) \to \operatorname{GL}_{n}(R), A \mapsto I_k \otimes A.\]
Here $I_k \otimes A$ denotes the Kronecker product of the identity matrix $I_k$ with $A$, which is the block diagonal matrix of $k$ copies of $A$. Let $(\Zbb_{>0}, |)$ denote the poset of positive numbers ordered by divisibility, the maps $\phi_{d, n}$ form a directed system over this poset. We write:
\begin{enumerate}
    \item The group $\operatorname{GL}_{\otimes}(R) = \operatorname{colim}_{(\Zbb_{>0}, |)} \operatorname{GL}_{n}(R)$ to be the colimit of this directed system.
    \item The group $\operatorname{GL}^p_{\otimes}(R) = \operatorname{colim}_{k \geq 0} \operatorname{GL}_{p^k}(R)$ to be the colimit of the sub-directed system on $\{1, p, p^2, ... \}$.
    \item For a multiplicatively closed $S \subset \Zbb_{> 0}$, $\operatorname{GL}^S_{\otimes}(R) = \operatorname{colim}_{S} \operatorname{GL}_{p^k}(R)$ to be the colimit of the sub-directed system on $S$.
\end{enumerate}
\end{definition}

Recall a path-connected space $X$ is \textit{simple} if it has abelian fundamental group and the canonical action of $\pi_1(X)$ on $\pi_n(X)$ (see Theorem 7.3.8 of \cite{Spanier1966}) is trivial for $n \geq 1$. Path-connected $H$-spaces (e.g., path-connected infinite loop spaces) are simple (see Theorem 7.3.9 of \cite{Spanier1966}).

\begin{lemma}\label{lem::plus_construction_des}
We have equivalences $\operatorname{BGL}_{\otimes}(R)^+ \simeq \tau_{>0} K(\operatorname{Free}_{\otimes}(R))$, $\operatorname{BGL}_{\otimes}^p(R)^+ \simeq \tau_{>0} K(\operatorname{Free}_{\otimes}^p(R))$, and $\operatorname{BGL}^S_{\otimes}(R)^+ \simeq \tau_{>0} K(\operatorname{Free}^S_{\otimes}(R))$. In particular, $\operatorname{BGL}_{\otimes}(R)^+,$ $\operatorname{BGL}_{\otimes}^p(R)^+,$ and $\operatorname{BGL}_{\otimes}^S(R)^+$ are infinite loop spaces and hence have abelian $\pi_1$ and are simple. 
\end{lemma}

Lemma~\ref{lem::plus_construction_des} follows from the general theory of $\operatorname{Aut}(\mathcal{S})$, which we briefly explain before the proof.

\begin{definition}
Let $(M, \cdot)$ be a commutative monoid. The translation category $\operatorname{tr}(M)$ of $M$ has objects as $M$ and morphisms $x \xrightarrow{\cdot z} y$ for any elements $x, y, z$ such that $y = x \cdot z$.
\end{definition}

\begin{definition}
Suppose $\mathcal{S}$ is a symmetric monoidal category such that $\operatorname{Aut}_{\mathcal{S}}(-): \operatorname{tr}(\operatorname{obj}(\mathcal{S})) \to \operatorname{Grp}$ is a well-defined functor from the translation category of $\operatorname{obj}(\mathcal{S})$ to groups. Here, the functor $\operatorname{Aut}_{\mathcal{S}}(-)$ sends $x \xrightarrow{\cdot z} y$ to tensoring with the identity map $\operatorname{id}_{z}$. The \textit{automorphism group} of $\mathcal{S}$ is
\[\operatorname{Aut}(\mathcal{S}) \coloneqq \operatorname{colim}_{x \in \operatorname{tr}(\operatorname{obj}(\mathcal{S}))} \operatorname{Aut}_{\mathcal{S}}(x).\]
\end{definition}

\begin{example}\label{exp::plus-construction}
The examples of interests to us are
\[\operatorname{GL}_{\otimes}(R) = \operatorname{Aut}(\operatorname{Free}_{\otimes}(R)), \operatorname{GL}^{p}_{\otimes}(R) = \operatorname{Aut}(\operatorname{Free}^p_{\otimes}(R)), \text{ and } \operatorname{GL}^{S}_{\otimes}(R) = \operatorname{Aut}(\operatorname{Free}^{S}_{\otimes}(R)).\]
Note also that $\operatorname{GL}(R) = \operatorname{Aut}(\operatorname{Free}_{\oplus}(R))$.  In fact, $\operatorname{GL}_{\otimes}(R)$ was used in \cite{c80f3555-900f-35b5-ab05-6d397e6a44e6}.
\end{example}
 
\begin{proof}[Proof of Lemma~\ref{lem::plus_construction_des}]
For a symmetric monoidal category $\mathcal{S}$ whose (i) morphisms are isomorphisms and (ii) translations are faithful, and (iii) has countably many objects, Proposition 3 of \cite{c80f3555-900f-35b5-ab05-6d397e6a44e6} or Theorem IV.4.10 of \cite{weibel2013k} implies that $[\operatorname{Aut}(\mathcal{S}), \operatorname{Aut}(\mathcal{S})]$ is the maximal perfect normal subgroup of $\operatorname{Aut}(\mathcal{S})$ and $\tau_{>0} K(\mathcal{S}) \simeq B\operatorname{Aut}(S)^+$. The Lemma now follows directly from Example~\ref{exp::plus-construction}.
\end{proof}

\subsection{Homological Stability} For our purposes, we also discuss a homological stability condition on commutative rings $R$ that will be relevant to our argument.
\begin{definition}\label{def::homological_stable}
    We say $R$ is \textit{homologically stable} if for each fixed $j$, the map $i_*: H_j(\operatorname{BGL}_k(R); \Zbb) \to H_j(\operatorname{BGL}_{k+1}(R); \Zbb)$ is an isomorphism for sufficiently high $k$.
\end{definition}
Examples of such $R$ include any Noetherian ring of finite Krull dimension, or more generally any ring with finite stable range in the sense of EI.1.5 in \cite{weibel2013k} (see also Corollary IV.1.14.1 of \cite{weibel2013k}, and Theorem 4.11 of \cite{Kallen1980}). For non-Noetherian examples, the reader is encouraged to look at Corollaries 5.14, 5.17, 5.21 and 5.22 of \cite{Mikkola2010}, which have homological stability because they
have finite stable rank. Definition~\ref{def::homological_stable} is different from May's definition of \textit{convergent} in Definition VII.5.1 of \cite{May1977EInfinityRingSpaces}, but this will be enough for our purposes.
\begin{remark}
    By \cite[Theorem 5]{Vasershtein1971}, the ring $R=C([0,1]^{\mathbb{N}})$ of continuous real-valued functions on $[0,1]^{\mathbb{N}}$ has infinite stable rank. Hence, for every $m \in \mathbb{N}$, there exists a unimodular $(m+1)$-tuple over $R$ that's not reducible. Such tuples represent non-trivial elements in the kernel of the stabilization map $H_1 (\mathrm{GL}_m(R);\Z)\to H_1 (\mathrm{GL}_{m+1}(R);\Z)$ for infinitely many $m$. Therefore, $R$ does not satisfy homological stability.
\end{remark}

\vspace{-5pt}
\section{The $p$-perfection Case}\label{sec::p_perfection}

The purpose of this section is to prove Theorem~\ref{thm::k_theory_p_perfection}. We will see how the strategies we used for this will carry over to Theorem~\ref{thm::k_theory_rationalization}.

\begin{lemma}\label{lem::p_diagram_commute}
Consider the diagram of groups displayed in Figure~\ref{eq::p_diagram}. This is (freely) homotopy commutative after applying the classifying space functor $B$.
\end{lemma}

\begin{figure}
    \centering
\begin{adjustbox}{width=0.475\textwidth}
\begin{tikzcd}
	1 & p & {p^2} & {p^k} \\
	{\operatorname{GL}_1(R)} & {\operatorname{GL}_1(R)} & {\operatorname{GL}_1(R)} & {...} \\
	{\operatorname{GL}_2(R)} & \vdots & \vdots & \ddots \\
	{\operatorname{GL}_3(R)} & {\operatorname{GL}_p(R)} & {\operatorname{GL}_p(R)} & \ddots \\
	\vdots & \vdots & \vdots & \ddots \\
	& {\operatorname{GL}_{2p}(R)} & {\operatorname{GL}_{p^2}(R)} & \ddots \\
	& \vdots & \ddots & \ddots
	\arrow["i"', from=2-1, to=3-1]
	\arrow["{\phi_{1,1p}}", from=2-1, to=4-2]
	\arrow["i", from=2-2, to=3-2]
	\arrow["{\phi_{1,1p}}", from=2-2, to=4-3]
	\arrow["i", from=2-3, to=3-3]
	\arrow[from=2-3, to=4-4]
	\arrow["i"', from=3-1, to=4-1]
	\arrow["{\phi_{2,2p}}", from=3-1, to=6-2]
	\arrow["i", from=3-2, to=4-2]
	\arrow[from=3-2, to=5-3]
	\arrow["i", from=3-3, to=4-3]
	\arrow[from=3-3, to=5-4]
	\arrow["i"', from=4-1, to=5-1]
	\arrow["{\phi_{3,3p}}"', from=4-1, to=7-2]
	\arrow["i", from=4-2, to=5-2]
	\arrow["{\phi_{p,p^2}}", from=4-2, to=6-3]
	\arrow["i", from=4-3, to=5-3]
	\arrow[from=4-3, to=6-4]
	\arrow["i", from=5-2, to=6-2]
	\arrow["i", from=5-3, to=6-3]
	\arrow["i", from=6-2, to=7-2]
	\arrow["{\phi_{2p,2p^2}}", from=6-2, to=7-3]
	\arrow["i", from=6-3, to=7-3]
	\arrow[from=6-3, to=7-4]
\end{tikzcd}
\end{adjustbox}
\caption{A diagram of groups corresponding to $p$-perfection. After applying homology, this is the same as $H_*(\operatorname{hocolim} \operatorname{BGL}(R)^+ \xrightarrow{\cdot p} \operatorname{BGL}(R)^+ \xrightarrow{\cdot p} ...; \Zbb)$}
\label{eq::p_diagram}
\end{figure}

\begin{remark}
    The diagram of groups in Figure~\ref{eq::p_diagram} does not commute.
\end{remark}

\begin{proof}
It is a general fact that two group homomorphisms $f, g: G \to H$ induce the same (free) homotopy class of maps $BG \to BH$ if and only if $f$ and $g$ differ by a conjugation in $H$. Although the diagram above does not commute, it does commute up to conjugation of elements in the codomain. For example, to check that the diagram here commutes up to conjugation:
\[\begin{tikzcd}
	{\operatorname{GL}_{n}(R)} & {\operatorname{GL}_{np}(R)} \\
	{\operatorname{GL}_{n+1}(R)} & {\operatorname{GL}_{(n+1)p}(R)}
	\arrow["{\phi_{1,p}}", from=1-1, to=1-2]
	\arrow["i", from=1-1, to=2-1]
	\arrow["i", from=1-2, to=2-2]
	\arrow["{\phi_{2,2p}}", from=2-1, to=2-2]
\end{tikzcd},\]
we see that given $A \in \operatorname{GL}_n(R)$, its image in this diagram is respectively

\[\begin{adjustbox}{width=0.5\textwidth}
$A_1 = \begin{pmatrix}
    A & & & \\
    & \ddots & & \\
    & & A & \\
    & & & & I_p
\end{pmatrix} \text{ and } A_2 = \begin{pmatrix}
    \begin{pmatrix}
        A & \\
        & 1
    \end{pmatrix} & &  \\
    & \ddots &  \\
    & & \begin{pmatrix}
        A & \\
        & 1
    \end{pmatrix}
\end{pmatrix}$
\end{adjustbox}.\]
We see that the $A_1 = A_2$ differ by some permutation of rows and columns, which can be described as the conjugation of some permutation matrix $P$ by linear algebra. Furthermore, this $P$ is independent of the element $A \in \operatorname{GL}_{n}(R)$. This shows that the diagram commutes up to conjugation. The other squares can be verified similarly.
\end{proof}

In particular, homotopy commutative diagrams of spaces give commutative diagrams after taking homology. This allows us to consider the colimit of  Figure~\ref{eq::p_diagram} after taking homology, which will be important in the following lemma.

\begin{lemma}\label{lem::prime_homology_iso}
For any commutative ring $R$, there is an explicit isomorphism $$H_*(\operatorname{BGL}_{\otimes}^p(R); \Zbb) \cong H_*(\operatorname{hocolim}(\operatorname{BGL}(R)^+ \xrightarrow{\cdot p} \operatorname{BGL}(R)^+ \xrightarrow{\cdot p} ...); \Zbb). $$
\end{lemma}

\begin{proof}
We first assume $R$ is homologically stable (Definition~\ref{def::homological_stable}). We write $L(p)$ to denote the diagram in Figure~\ref{eq::p_diagram} and $I$ to be the indexing diagram. Observe that
\[H_*(\operatorname{hocolim}(\operatorname{BGL}(R)^+ \xrightarrow{\cdot p} \operatorname{BGL}(R)^+ \xrightarrow{\cdot p} ...); \Zbb) \cong \operatorname{colim}_{I} H_*(L(p); \Zbb).\]
This is because homology commutes with filtered colimits, plus-constructions are homology isomorphisms, and $\cdot p$ on $\operatorname{BGL}(R)^+$ is induced by sending a matrix $A$ to $\bigoplus_{i=1}^p A$, which is exactly the block diagonal matrix of $p$-copies of $A$. Thus, the effect of multiplication by $p$ on the colimit diagram of $H_*(\operatorname{GL}(R))$ is portrayed exactly by applying $H_*(-)$ to Figure~\ref{eq::p_diagram}.\\

Now observe that the colimit of applying $H_*$ to the following diagram gives $H_*(\operatorname{BGL}_{\otimes}^p(R); \Zbb)$: 
\begin{equation}\label{eq::sub_diagram_p}%
\begin{tikzcd}[ampersand replacement=\&]
	{\operatorname{GL}_{p^k}(R)} \&\& {\operatorname{GL}_{p^{k+1}}(R)} \&\& {\operatorname{GL}_{p^{k+2}}(R)} \& {...}
	\arrow["{\phi_{p^k, p^{k+1}}}", from=1-1, to=1-3]
	\arrow["{\phi_{p^{k+1}, p^{k+2}}}", from=1-3, to=1-5]
	\arrow[from=1-5, to=1-6]
\end{tikzcd}\end{equation}
Furthermore, (\ref{eq::sub_diagram_p}) can be identified as the full sub-diagram $\Sigma_{k}$ of Figure~\ref{eq::p_diagram} consisting of $\operatorname{GL}_{p^k}(R)$ on column labeled $1$, $\operatorname{GL}_{p^{k+1}}(R)$ on column labeled $p$, $\operatorname{GL}_{p^{k+2}}(R)$ on column labeled $p^2$, and more generally $\operatorname{GL}_{p^{k+\ell}(R)}$ on column labeled $p^{\ell}$ for $\ell \geq 0$.\\

This shows in particular by cofinality that we can write
\[\operatorname{colim}_I H_*(L(p); \Zbb) \cong \operatorname{colim}_{k} H_*(\operatorname{colim} \Sigma_k; \Zbb).\]
Visually, we are turning Figure~\ref{eq::p_diagram} into looking at vertical maps between sub-diagrams given by $\Sigma_k$'s as $k$ increases. Now we know each $H_*(\operatorname{colim} \Sigma_k; \Zbb) \cong H_*(\operatorname{BGL}^p_{\otimes}(R); \Zbb)$, so we have that
\[\operatorname{colim}_I H_*(L(p); \Zbb) \cong \operatorname{colim}_{k} H_*(\operatorname{BGL}^p_{\otimes}(R); \Zbb).\]
Since $R$ is homologically stable, each transition map in the colimit on the right is an isomorphism as it is cofinally an isomorphism. It follows that
\[\operatorname{colim}_I H_*(L(p); \Zbb) \cong  H_*(\operatorname{BGL}^p_{\otimes}(R); \Zbb).\]

For a general commutative ring $R$, we note that $R$ is a filtered direct limit of its subrings $R_i$ that are finitely generated $\Zbb$-algebras and hence Noetherian of finite Krull dimension. Since Noetherian rings of finite Krull dimension are homologically stable, the isomorphism holds on the $R_i$'s. To show the isomorphism on $R$ holds, it suffices to check that the left side $\operatorname{colim}_{k} H_*(\operatorname{BGL}_{p^k}(-); \Zbb)$ and right side $\operatorname{colim} H_*(\operatorname{BGL}(-)^+; \Zbb) \cong \operatorname{colim} H_*(\operatorname{BGL}(-); \Zbb)$ commute with direct limit of rings. This is true for the left side as $\operatorname{GL}_n(-)$ commutes with direct limits, the classifying space functor $B$ commutes with filtered colimits, homology commutes with filtered colimits, and colimits commute with one another. This is also true for the right side as $\operatorname{GL}(-)$ is a filtered colimit of $\operatorname{GL}_n(-)$'s and colimits commute.
\end{proof}

We now prove Theorem~\ref{thm::k_theory_p_perfection}, which will follow directly from the following.
\begin{theorem}\label{thm::p_perfection}
There is an equivalence
\[(\operatorname{BGL}_{\otimes}^p(R))^+ = (\operatorname{hocolim}_{k} \operatorname{BGL}_{p^k}(R))^+ \simeq \operatorname{hocolim}(\operatorname{BGL}(R)^+ \xrightarrow{\cdot p} \operatorname{BGL}(R)^+ \xrightarrow{\cdot p} \operatorname{BGL}(R)^+ \xrightarrow{\cdot p} ...).\]
Using Lemma~\ref{lem::plus_construction_des} and the plus-construction description for $K(R)$, it follows that
\[\tau_{>0} K_{\otimes}^p(R) \simeq (\tau_{>0} K(R))\left[\frac{1}{p}\right].\]
\end{theorem}

\begin{proof}
Consider the homotopy commutative ladder diagram
\[
\begin{tikzcd}
\mathrm{BGL}_{p^0}(R) \ar[r,"I_p \otimes -"] \ar[d,"s_0"] &
\mathrm{BGL}_{p^1}(R) \ar[r,"I_p \otimes -"] \ar[d,"s_1"] &
\mathrm{BGL}_{p^2}(R) \ar[r,"I_p \otimes -"] \ar[d,"s_2"] & \cdots \\
\mathrm{BGL}(R)^+ \ar[r,"\cdot p"] &
\mathrm{BGL}(R)^+ \ar[r,"\cdot p"] &
\mathrm{BGL}(R)^+ \ar[r,"\cdot p"] & \cdots
\end{tikzcd}
\]
where $s_i: \operatorname{BGL}_{p^k}(R) \to \operatorname{BGL}(R) \to \operatorname{BGL}(R)^+$ is the stabilization map composed before the plus construction map. In general, for a map of homotopy colimit diagrams that is only homotopy commutative, one needs the information of higher homotopy coherence to induce a map between their homotopy colimits. Fortunately, this is not required for sequential diagrams. By Lemma A.10 and its preceding discussions in \cite{GalatiusRandalWilliams2017}, an independent choice of homotopy $H_i$ for each $s_{i+1} \circ I_p \otimes - \simeq \cdot p \circ s_{i}$ is enough to give a map $\phi: \operatorname{hoclim}_k \operatorname{BGL}_{p^k}(R) \to \operatorname{hocolim}(\operatorname{BGL}(R)^+ \xrightarrow{\cdot p} \operatorname{BGL}(R)^+ \xrightarrow{\cdot p} ...) = (\operatorname{BGL}(R)^+)[\frac{1}{p}]$. Furthermore, the property on whether or not this map is an isomorphism on integral homology only depends on the underlying homotopy commutative diagram.\\

Under the identifications and notations in Lemma~\ref{lem::prime_homology_iso}, the map $\phi$ induces in homology:
\[\operatorname{colim}_k H_*(\operatorname{BGL}_{p^k}(R)) \cong H_*(\operatorname{hocolim}_k \operatorname{BGL}_{p^k}(R))\xrightarrow{\phi_*} \operatorname{hocolim}(\operatorname{BGL}(R)^+ \xrightarrow{\cdot p} \operatorname{BGL}(R)^+ \xrightarrow{\cdot p} ...).\]
This can be identified as sending $\operatorname{colim}_k H_*(\operatorname{BGL}_{p^k}(R))$ directly to $H_*(\operatorname{colim} \Sigma_1; \Zbb)$. Visually, this means the colimit fits in as the diagonal of Figure~\ref{eq::p_diagram}. Thus, by Lemma~\ref{lem::prime_homology_iso}, the map realizes the isomorphism
\[\operatorname{colim}_k H_*(\operatorname{BGL}_{p^k}(R)) \cong H_*(\operatorname{colim} \Sigma_1; \Zbb) \cong \operatorname{colim}_{k} H_*(\operatorname{colim} \Sigma_k; \Zbb) \cong \operatorname{colim}_I H_*(L(p); \Zbb) \]
\[\cong \operatorname{hocolim}(\operatorname{BGL}(R)^+ \xrightarrow{\cdot p} \operatorname{BGL}(R)^+ \xrightarrow{\cdot p} ...).\]
Now we have a map
\[\operatorname{hocolim}_k \operatorname{BGL}_{p^k}(R) \to \operatorname{hocolim}(\operatorname{BGL}(R)^+ \xrightarrow{\cdot p} \operatorname{BGL}(R)^+ \xrightarrow{\cdot p} \operatorname{BGL}(R)^+ \xrightarrow{\cdot p} ...) \]
that is an integral homology isomorphism and clearly sends to the commutator subgroup of $\pi_1(\operatorname{hocolim}_k \operatorname{BGL}_{p^k}(R))$ to  $0$, as the right side has abelian fundamental group. The universal property of the plus construction gives a map
\[(\operatorname{hocolim}_{k} \operatorname{BGL}_{p^k}(R))^+ \to \operatorname{hocolim}(\operatorname{BGL}(R)^+ \xrightarrow{\cdot p} \operatorname{BGL}(R)^+ \xrightarrow{\cdot p} \operatorname{BGL}(R)^+ \xrightarrow{\cdot p} ...).\]
It is a general fact that a map $f: X \to Y$ of simple spaces that is both an isomorphism on integral homology and $\pi_1$ is a homotopy equivalence (e.g., Lemma 1.1 of \cite{WAGONER1972349}). Since both sides are infinite loop spaces (the left side follows from Lemma~\ref{lem::plus_construction_des}, the right side is because $\operatorname{BGL}(R)^+$ is an infinite loop space) and hence simple, it suffices to check that this is also a $\pi_1$-isomorphism to show this is an equivalence. Clearly, the right side has an abelian fundamental group as $\pi_1$ commutes with filtered colimits, and the left side has an abelian fundamental group from Lemma~\ref{lem::plus_construction_des}. Since the homology isomorphism is, in particular, an $H_1$-isomorphism, it follows that they are equivalent.
\end{proof}

\section{The Rationalization Case}\label{sec::rationalization}

In this section, we will prove Theorem~\ref{thm::k_theory_rationalization} using strategies similar to the $p$-perfection case.

\begin{definition}\label{def::colimit_notation}
   Let $L$ denote the colimit diagram for $\operatorname{GL}(R)$ and $n \in \Zbb_{>0}$. The diagram $L \xrightarrow{\cdot n} L$ is obtained by adding the sequence $\phi_{k, nk}: \operatorname{GL}_{k}(R) \to \operatorname{GL}_{nk}(R)$ for all $k$ to the diagram $L \sqcup L$.
\end{definition}

The following Figure~\ref{fig:L} shows an example of $L \xrightarrow{\cdot 2} L$ for when $n = 2$.
\begin{figure}[!htb]
    \centering

 \begin{adjustbox}{width=0.25\textwidth}
\begin{tikzcd}
	{\operatorname{GL}_1(R)} & {\operatorname{GL}_1(R)} \\
	{\operatorname{GL}_2(R)} & {\operatorname{GL}_2(R)} \\
	{\operatorname{GL}_3(R)} & {\operatorname{GL}_3(R)} \\
	{\operatorname{GL}_4(R)} & {\operatorname{GL}_4(R)} \\
	\vdots & \vdots
	\arrow["i"', from=1-1, to=2-1]
	\arrow["{\phi_{1,2}}"{description}, from=1-1, to=2-2]
	\arrow["i", from=1-2, to=2-2]
	\arrow["i"', from=2-1, to=3-1]
	\arrow["{\phi_{2,4}}"{description}, from=2-1, to=4-2]
	\arrow["i", from=2-2, to=3-2]
	\arrow["i"', from=3-1, to=4-1]
	\arrow["{\phi_{3,6}}"{description}, from=3-1, to=5-2]
	\arrow["i", from=3-2, to=4-2]
	\arrow["i"', from=4-1, to=5-1]
	\arrow["{\phi_{4,8}}"{description}, from=4-1, to=5-2]
	\arrow["i", from=4-2, to=5-2]
\end{tikzcd}
\end{adjustbox}

\caption{Example illustrating $L \xrightarrow{\cdot 2} L$ in Definition~\ref{def::colimit_notation}.}
\label{fig:L}
\end{figure}

\begin{lemma}\label{lem::rational_commute}
Consider the diagram in Figure~\ref{eq::rationalization} using the notations in Definition~\ref{def::colimit_notation}. To elaborate, on the $n$-th column from $0, 1, 2, ...$, we put a copy of $L$ for each positive number with $n$ prime factors including multiplicity. For a copy of $L$ on row $n+1$ and a copy of $L$ on row $n$, we draw a map with $\cdot p$ between them if and only if their respective numbers differ by that prime $p$.\\

Figure~\ref{eq::rationalization} is (freely) homotopy commutative after applying the classifying space functor $B$.
\end{lemma}

\begin{figure}
    \centering
\begin{adjustbox}{width=0.6\textwidth}
\begin{tikzcd}
	{\textcolor{red}{\text{0 Prime Factor}}} & {\textcolor{red}{\text{1 Prime Factor}}} & {\textcolor{red}{\text{2 Prime Factors}}} & {\textcolor{red}{\text{...}}} \\
	&& {\underbrace{L}_{\textcolor{red}{(4)}}} & \dots \\
	& {\underbrace{L}_{\textcolor{red}{(2)}}} \\
	{\underbrace{L}_{\textcolor{red}{(1)}}} && {\underbrace{L}_{\textcolor{red}{(6)}}} & \dots \\
	& {\underbrace{L}_{\textcolor{red}{(3)}}} \\
	& {\underbrace{L}_{\textcolor{red}{(5)}}} & {\underbrace{L}_{\textcolor{red}{(9)}}} & \dots \\
	& \vdots & {\underbrace{L}_{\textcolor{red}{(10)}}} & \ddots
	\arrow["{\cdot 2}", from=3-2, to=2-3]
	\arrow["{\cdot 3}", from=3-2, to=4-3]
	\arrow["{\cdot 5}"{description}, dashed, from=3-2, to=7-3]
	\arrow["{\cdot 2}", from=4-1, to=3-2]
	\arrow["{\cdot 3}", from=4-1, to=5-2]
	\arrow["{\cdot 5}", from=4-1, to=6-2]
	\arrow[shift right=5, from=4-1, to=7-2]
	\arrow["{\cdot 2}", from=5-2, to=4-3]
	\arrow["{\cdot 3}", from=5-2, to=6-3]
	\arrow["\cdot2", from=6-2, to=7-3]
\end{tikzcd}.
\end{adjustbox}
\caption{A diagram of groups corresponding to the rationalization case. Details are explained in Lemma~\ref{lem::rational_commute}.}
\label{eq::rationalization}
\end{figure}

\begin{proof}
    We first note that applying $B$ to $L \xrightarrow{\cdot p} L$ gives a freely homotopy commutative diagram (i.e., each square in the diagram commutes up to a free homotopy), as it recurs back to the case of Lemma~\ref{lem::p_diagram_commute}.
    
    It then suffices for us to show that for primes $a, b, c, d$ such that $ab = cd$, applying $B$ to the diagram
\begin{equation}\label{eq::commutative_square}\begin{tikzcd}[ampersand replacement=\&]
	{\underbrace{L}_{(X)}} \& {\underbrace{L}_{(Y)}} \\
	{\underbrace{L}_{(Z)}} \& {\underbrace{L}_{(W)}}
	\arrow["{\cdot a}", from=1-1, to=1-2]
	\arrow["{\cdot c}"', from=1-1, to=2-1]
	\arrow["{\cdot b}", from=1-2, to=2-2]
	\arrow["{\cdot d}"', from=2-1, to=2-2]
\end{tikzcd}\end{equation}
gives a freely homotopy commutative diagram. Note this includes the case when $a = c, b = d$. Indeed, suppose $A \in \operatorname{GL}_{n}(R)$ is at the copy of $L$ with $(X)$ underneath. Evaluating $A$ at a square in Diagram (\ref{eq::commutative_square}) starting at $\operatorname{GL}_n(R)$ in $\underbrace{L}_{(X)}$ must give two $N \times N$ matrices: $\psi_{a, b}(A)$, along $L \xrightarrow{\cdot a} L \xrightarrow{\cdot b} L$, and $\psi_{c, d}(A)$, along $L \xrightarrow{\cdot c} L \xrightarrow{\cdot d} L$. In terms of block diagonal matrices, $\psi_{a,b}(A)$ has $ab$ copies of $A$ and $N - ab n$ remaining copies of $1$. The matrix $\psi_{c, d}(A)$ has $cd$ copies of $A$ and $N - cd n$ remaining copies of $1$. Note that here we may have extra copies of $1$ because the diagram $L \xrightarrow{\cdot n} L$ has vertical arrows going down (see Figure~\ref{fig:L}) that could be part of the composition when checking the commutativity.

Since $ab = cd$ and the positions of $A$ in the block diagonal matrices are fixed no matter what $A \in \operatorname{GL}_n(R)$ is, we see that $\psi_{a, b}(A)$ and $\psi_{c, d}(A)$ differ by the conjugation of the same permutation matrix. Thus, the same reasoning as in Lemma~\ref{lem::p_diagram_commute} shows that applying $B$ to (\ref{eq::commutative_square}) gives a freely homotopy commutative diagram.
\end{proof}

\vspace{-2pt}
We now obtain an analog of Lemma~\ref{lem::prime_homology_iso} for the rationalization case.
\begin{lemma}\label{lem::rational_homology_iso}
  There is an explicit isomorphism $H_*(\operatorname{BGL}_{\otimes}(R); \Zbb) \cong H_*(\operatorname{hocolim}_{(\Zbb_{>0}, |)} \operatorname{BGL}(R)^+)$. Here, the transition maps in the colimit are given by $\cdot k$ for an arrow $d \xrightarrow{\cdot k} n$ in the colimit diagram.  
\end{lemma}

\begin{proof}
First, we assume that $R$ is homologically stable. We write $F$ to denote Figure~\ref{eq::rationalization}. Observe that $\operatorname{colim} H_*(F)$ is the same as $H_*(\operatorname{hocolim}_{(\Zbb_{>0}, |)} \operatorname{BGL}(R)^+)$ after commuting the filtered colimit with $H_*$, getting rid of the plus construction, and expanding the colimit diagram.

Now observe that the colimit of applying $H_*$ to the following diagram gives $H_*(\operatorname{BGL}_{\otimes}(R); \Zbb)$ by cofinality.
\begin{equation}\label{eq:n_factorial}
\begin{adjustbox}{width=0.55\textwidth}
\begin{tikzcd}
	& {\text{\textcolor{red}{1 prime factor}}} & {\text{\textcolor{red}{2 prime factors}}} & {\textcolor{red}{...}} \\
	& {\operatorname{GL}_{2n!}(R)} & {\operatorname{GL}_{4n!}(R)} & {...} \\
	{\operatorname{GL}_{n!}(R)} & {\operatorname{GL}_{3n!}(R)} & {\operatorname{GL}_{6n!}(R)} & {...} \\
	& {\operatorname{GL}_{5n!}(R)} & {\operatorname{GL}_{9n!}(R)} & {...} \\
	& \vdots & \vdots & \ddots
	\arrow["{\cdot 2}"{description}, from=2-2, to=2-3]
	\arrow["{\cdot 3}"{description}, from=2-2, to=3-3]
	\arrow["{\cdot 2}", from=3-1, to=2-2]
	\arrow["{\cdot 3}", from=3-1, to=3-2]
	\arrow["{\cdot 5}"', from=3-1, to=4-2]
	\arrow["{\cdot 2}"{description}, from=3-2, to=3-3]
	\arrow["{\cdot 3}"{description}, from=3-2, to=4-3]
\end{tikzcd}
\end{adjustbox}
\end{equation}
Furthermore, (\ref{eq:n_factorial}) can be identified as the full subdiagram $\Sigma_n$ in Figure~\ref{eq::rationalization} over the terms $\operatorname{GL}_{kn!}(R)$ in the copy of $L$ labeled with $(k)$ underneath, for all $k \geq 1$. This shows in particular that $\operatorname{colim} H_*(F; \Zbb) \cong \operatorname{colim}_n \operatorname{colim} H_*(\Sigma_n; \Zbb)$. Furthermore, each transition map in the outer colimit is an isomorphism as $R$ is homologically stable, thus we have that $$\operatorname{colim}_n \operatorname{colim} H_*(\Sigma_n; \Zbb) \cong H_*(\operatorname{BGL}_{\otimes}(R); \Zbb).$$
The general case now follows similarly as in Lemma~\ref{lem::prime_homology_iso} as both sides commute with filtered colimits, and every ring $R$ is a direct limit of homologically stable rings.
\end{proof}

We now use Lemma~\ref{lem::rational_homology_iso} to show an equivalence of plus construction spaces.
\begin{lemma}\label{lem::rational_tensor}
There is an equivalence $\operatorname{BGL}_{\otimes}(R)^+ \simeq \operatorname{hocolim}_{(\Zbb, |)} \operatorname{BGL}(R)^+ = \operatorname{BGL}(R)^+ \otimes \Qbb$.
\end{lemma}

\begin{proof}
Consider the homotopy commutative ladder diagram
\[
\begin{tikzcd}
\mathrm{BGL}_{1}(R) \ar[r,"I_2 \otimes -"] \ar[d,"s_0"] &
\mathrm{BGL}_{2!}(R) \ar[r,"I_3 \otimes -"] \ar[d,"s_1"] &
\mathrm{BGL}_{3!}(R) \ar[r,"I_4 \otimes -"] \ar[d,"s_2"] & \cdots \\
\mathrm{BGL}(R)^+ \ar[r,"\cdot 2"] &
\mathrm{BGL}(R)^+ \ar[r,"\cdot 3"] &
\mathrm{BGL}(R)^+ \ar[r,"\cdot 4"] & \cdots
\end{tikzcd}
\]
Using Lemma A.10 of \cite{GalatiusRandalWilliams2017} again (as in Theorem~\ref{thm::p_perfection}), we have a map $\phi: \operatorname{hocolim}_n \operatorname{BGL}_{n!}(R) \to \operatorname{hocolim}(\operatorname{BGL}(R)^+ \xrightarrow{\cdot 2} \operatorname{BGL}(R)^+ \xrightarrow{\cdot 3} ...) \cong \operatorname{BGL}(R)^+ \otimes \Qbb$. The map in homology $\phi$ induces
\[\operatorname{colim}_{n} H_*(\operatorname{BGL}_{n!}(R)) \to H_*(\operatorname{hocolim}(\operatorname{BGL}(R)^+ \xrightarrow{\cdot 2} \operatorname{BGL}(R)^+ \xrightarrow{\cdot 3} ...); \Zbb)\]
is identified with the map
\[\operatorname{colim}_{n} H_*(\operatorname{BGL}_{n!}(R)) \xrightarrow{\cong}_f \operatorname{colim} H_*(\Sigma_1; \Zbb) \xrightarrow{\cong}_g \operatorname{colim}_n \operatorname{colim} H_*(\Sigma_n; \Zbb) \xrightarrow{\cong}_h \operatorname{colim} H_*(F; \Zbb) \cong \]
\[H_*(\operatorname{hocolim}(\operatorname{BGL}(R)^+ \xrightarrow{\cdot 2} \operatorname{BGL}(R)^+ \xrightarrow{\cdot 3} ...); \Zbb)\]
where the $f$ is given by identifying the domain as a cofinal subdiagram in the second, $g$ is the natural inclusion into the colimit and is an isomorphism by Lemma~\ref{lem::rational_homology_iso}, the map $h$ is an isomorphism in Lemma~\ref{lem::rational_homology_iso}, and the final map is an isomorphism by cofinal identifications (in the backward direction).

Thus, we now have a map
\[ \operatorname{hocolim}_n \operatorname{BGL}_{n!}(R) \to \operatorname{BGL}(R)^+ \otimes \Qbb\]
that is a homology isomorphism. Now, $\operatorname{BGL}_{\otimes}(R)^+$ is an infinite loop space by Lemma~\ref{lem::plus_construction_des} and hence simple, and $\operatorname{BGL}(R)^+ \otimes \Qbb$ is clearly also simple. Thus, the same lines of arguments as in the end of Theorem~\ref{thm::p_perfection} would conclude that $\operatorname{BGL}_{\otimes}(R)^+ \simeq \operatorname{hocolim}_{(\Zbb, |)} \operatorname{BGL}_n(R)$.
\end{proof}

Finally, we prove Theorem~\ref{thm::k_theory_rationalization} as follows.

\begin{proof}[Proof of Theorem~\ref{thm::k_theory_rationalization}]
Lemma~\ref{lem::rational_tensor} and Lemma~\ref{lem::plus_construction_des} show that $\tau_{> 0} K(\operatorname{Free}_{\otimes}(R))$ is equivalent to the rationalization of $\tau_{> 0} K(R)$. To conclude that they are equivalent as infinite loop spaces, we note that their corresponding connective spectra have the same homotopy groups and are all rational. It is a well-known fact that a rational spectrum $X$ is equivalent $\bigvee_{n} H\pi_n(X)$. Thus, their corresponding connective spectra are equivalent, and hence the underlying infinite loop space structures.
\end{proof}

\section{The General Case of Multiplicative Localization}\label{sec::k_theory_localization}

In this section, we will prove Theorem~\ref{thm::k_theory_localization} by generalizing the strategies we have used earlier. Fix $S \subseteq \Zbb_{> 0}$ to be a multiplicatively closed subset. We define a diagram $F_{S}$ as the sub-diagram of Figure~\ref{eq::rationalization} consisting only of $L$'s labeled with $s \in S$ and maps between them.

\begin{lemma}\label{lem::localization_iso}
   There is an explicit isomorphism $H_*(\operatorname{BGL}_{\otimes}^S(R); \Zbb) \cong H_*(\operatorname{hocolim}_{S} \operatorname{BGL}(R)^+; \Zbb)$.
\end{lemma}

\begin{proof}
    We first assume $R$ is homologically stable. Observe that $\operatorname{colim} H_*(F_S)$ is the same as $H_*(\operatorname{hocolim}_{S} \operatorname{BGL}(R)^+)$ after commuting the filtered colimit with $H_*$, getting rid of the plus construction, and expanding the colimit diagram. Now order the elements of $S$ in increasing order as $1 = a_1 < a_2 < a_3 < a_4 < ... < a_n < ...$. Write $A_n = \prod_{i = 1}^n a_i$. For each $n$, we consider the full subdiagram $\Sigma_n$ in $F_{S}$ on the terms $\operatorname{GL}_{a_i A_n}(R)$ in the copy of $L$ labeled with $(a_i)$ underneath, for $i \geq 1$. This diagram $\Sigma_n$ can be naturally identified as a cofinal diagram in $\Sigma_1$, so
    \[\operatorname{colim} H_*(\Sigma_n; \Zbb) \cong \operatorname{colim} H_*(\Sigma_1; \Zbb) \cong H_*(\operatorname{BGL}_{\otimes}^{S}; \Zbb).\]
    Furthermore, we can write 
    \[\operatorname{colim} H_*(F_S; \Zbb) \cong \operatorname{colim}_n \operatorname{colim} H_*(\Sigma_n; \Zbb) \cong \operatorname{colim}_n H_*(\operatorname{BGL}_{\otimes}^{S}; \Zbb). \]
    Since $R$ is homologically stable, the transition maps in the middle term above are all isomorphisms. Thus, we have that $\operatorname{colim} H_*(F_S; \Zbb) \cong H_*(\operatorname{BGL}_{\otimes}^S(R); \Zbb)$. The general case now follows similarly as in Lemma~\ref{lem::prime_homology_iso} as both sides commute with filtered colimits, and every ring $R$ is a direct limit of homologically stable rings.
\end{proof}

Finally, Theorem~\ref{thm::k_theory_localization} will follow from the following statement.
\begin{lemma}
There is an equivalence $\operatorname{BGL}_{\otimes}^S(R)^+ \simeq \operatorname{hocolim}_{S} \operatorname{BGL}(R)^+ = \operatorname{BGL}(R)^+[S^{-1}]$.
\end{lemma}

\begin{proof}
Order the elements of $S$ in increasing order and write them as $1 = a_1 < a_2 < a_3 < a_4 < ... < a_n < ...$. Write $A_n = \prod_{i=1}^n a_i$. We can consider the homotopy commutative ladder diagram
\[
\begin{tikzcd}
\mathrm{BGL}_{A_1}(R) \ar[r,"I_{s_2} \otimes -"] \ar[d,"s_0"] &
\mathrm{BGL}_{A_2}(R) \ar[r,"I_{s_3} \otimes -"] \ar[d,"s_1"] &
\mathrm{BGL}_{A_3}(R) \ar[r,"I_{s_4} \otimes -"] \ar[d,"s_2"] & \cdots \\
\mathrm{BGL}(R)^+ \ar[r,"\cdot s_2"] &
\mathrm{BGL}(R)^+ \ar[r,"\cdot s_3"] &
\mathrm{BGL}(R)^+ \ar[r,"\cdot s_4"] & \cdots
\end{tikzcd}
.\]
Lemma A.10 of \cite{GalatiusRandalWilliams2017} again gives us a map $\phi: \operatorname{hocolim}_n \operatorname{BGL}_{A_n}(R) \to \operatorname{hocolim}(\operatorname{BGL}(R)^+ \xrightarrow{\cdot a_2} \operatorname{BGL}(R)^+ \xrightarrow{\cdot a_3} ...)$. After taking homology, we have a commutative diagram
\[\begin{tikzcd}
	{H_*(\operatorname{hocolim}_n \operatorname{BGL}_{A_n}(R); \Zbb)} & {H_*(\operatorname{hocolim}(\operatorname{BGL}(R)^+ \xrightarrow{\cdot a_2} \operatorname{BGL}(R)^+ \xrightarrow{\cdot a_3} ...); \Zbb)} \\
	{H_*(\operatorname{BGL}_{\otimes}^S(R); \Zbb)} & {H_*(\operatorname{hocolim}_{S} \operatorname{BGL}(R)^+; \Zbb)}
	\arrow["{\phi_*}", from=1-1, to=1-2]
	\arrow["{\text{cofinal}}"', from=1-1, to=2-1]
	\arrow["{\text{cofinal}}", from=1-2, to=2-2]
	\arrow["{\text{Map in Lemma~\ref{lem::localization_iso}}}"', from=2-1, to=2-2]
\end{tikzcd}\]
where the vertical arrows are isomorphisms, and the bottom arrow is an isomorphism by Lemma~\ref{lem::localization_iso}. This shows $\phi$ induces an integral homology isomorphism. Since $\operatorname{BGL}_{\otimes}^S(R)^+$ is simple by Lemma~\ref{lem::plus_construction_des} and $\operatorname{BGL}(R)^+[S^{-1}]$ is simple. The same lines of arguments as in the end of Theorem~\ref{thm::k_theory_p_perfection} would conclude that $\operatorname{BGL}_{\otimes}^S(R)^+ \simeq \operatorname{hocolim}_{S} \operatorname{BGL}(R)^+ = \operatorname{BGL}(R)^+[S^{-1}]$.
\end{proof}

\begin{remark}
    Our argument also adapts to other scenarios of interest. For example, if one restricts the definition of $\operatorname{Free}_{\otimes}(\mathbb{C})$ to only unitary matrices, similar statements would hold for the unitary analog of algebraic $K$-theory (i.e., the homotopy groups of $(BU^{\delta})^+$, where $U^{\delta}$ is the underlying discrete group of $U$). The homological stability condition would be satisfied due to \cite{Sah1986}.
\end{remark}

\bibliographystyle{alpha}
\bibliography{references.bib}

@article{c80f3555-900f-35b5-ab05-6d397e6a44e6,
 ISSN = {00029939, 10886826},
 URL = {http://www.jstor.org/stable/2043975},
 author = {C.A. Weibel},
 journal = {Proceedings of the American Mathematical Society},
 number = {1},
 pages = {1--7},
 publisher = {American Mathematical Society},
 title = {{$K$-Theory of Azumaya Algebras}},
 volume = {81},
 year = {1981}
}

@book{bass1968algebraic,
  title={Algebraic K-theory},
  author={Bass, H.},
  isbn={9780805306606},
  lccn={lc68024366},
  series={Hyman Bass},
  year={1968},
  publisher={W. A. Benjamin}
}

@book{May1977EInfinityRingSpaces,
  author    = {May, J.P.},
  title     = {{E$_\infty$ Ring Spaces and E$_\infty$ Ring Spectra}},
  series    = {Lecture Notes in Mathematics},
  volume    = {577},
  publisher = {Springer-Verlag},
  address   = {Berlin, Heidelberg},
  year      = {1977},
  isbn      = {978-3-540-08136-4},
  doi       = {10.1007/BFb0097608},
  url       = {https://link.springer.com/book/10.1007/BFb0097608},
  note      = {Lecture Notes in Mathematics, volume 577. With contributions by Frank Quinn, Nigel Ray, and Jørgen
Tornehave}
}

@article{ramzi2025pperfectiongroupcompletionmathbbeinftymonoids,
  author  = {Ramzi, Maxime and Yakerson, Maria},
  title   = {$p$-perfection and group completion of {$\mathbb{E}_\infty$}-monoids},
  journal = {Documenta Mathematica},
  year    = {2026},
  doi     = {10.4171/DM/1086},
  note    = {Published online first}
}

@book{weibel2013k,
  title={The $K$-book: An Introduction to Algebraic $K$-theory},
  author={Weibel, C.A.},
  isbn={9780821891322},
  lccn={2012039660},
  series={Graduate Studies in Mathematics},
  url={https://books.google.com/books?id=Ja8xAAAAQBAJ},
  year={2013},
  publisher={American Mathematical Society}
}

@article{Segal1974,
  author    = {G. Segal},
  title     = {Categories and cohomology theories},
  journal   = {Topology},
  volume    = {13},
  year      = {1974},
  pages     = {293--312},
  issn      = {0040-9383},
  doi       = {10.1016/0040-9383(74)90022-6},
  mrnumber  = {0353298},
}

@book{may_operads,
  author    = {May, J.P.},
  title     = {The Geometry of Iterated Loop Spaces},
  series    = {Lecture Notes in Mathematics},
  volume    = {271},
  publisher = {Springer-Verlag},
  address   = {Berlin, Heidelberg},
  year      = {1972},
  isbn      = {978-3-540-05904-2, 978-3-540-37603-3},
  doi       = {10.1007/BFb0067491}
}

@article{may_thomason_uniqueness,
title = {The uniqueness of infinite loop space machines},
journal = {Topology},
volume = {17},
number = {3},
pages = {205-224},
year = {1978},
issn = {0040-9383},
doi = {https://doi.org/10.1016/0040-9383(78)90026-5},
url = {https://www.sciencedirect.com/science/article/pii/0040938378900265},
author = {J.P. May and R. Thomason}
}

@article{may2009preciselyeinftyringspaces,
    author={J. P. May},
     title={{What precisely are $E_{\infty}$ ring spaces and $E_{\infty}$ ring spectra?}}, 
    journal = {Geometry \& Topology Monographs},
    year = {2009},
    volume={16},
    number={09}
}

@article{MAY19821,
title = {Multiplicative infinite loop space theory},
journal = {Journal of Pure and Applied Algebra},
volume = {26},
number = {1},
pages = {1-69},
year = {1982},
issn = {0022-4049},
doi = {https://doi.org/10.1016/0022-4049(82)90029-9},
url = {https://www.sciencedirect.com/science/article/pii/0022404982900299},
author = {J.P. May}
}

@article{Gepner_2015,
   title={Universality of multiplicative infinite loop space machines},
   volume={15},
   ISSN={1472-2747},
   url={http://dx.doi.org/10.2140/agt.2015.15.3107},
   DOI={10.2140/agt.2015.15.3107},
   number={6},
   journal={Algebraic \& Geometric Topology},
   publisher={Mathematical Sciences Publishers},
   author={Gepner, D. and Groth, M. and Nikolaus, T.},
   year={2015},
   month=Dec, pages={3107–3153} }

@misc{ji2026quantumcellularautomatagroup,
      title={{Quantum Cellular Automata: The Group, the Space, and the Spectrum}}, 
      author={M. Ji and B. Yang},
      year={2026},
      eprint={2602.16572},
      archivePrefix={arXiv},
      primaryClass={math.AT},
      url={https://arxiv.org/abs/2602.16572}, 
}

@InProceedings{quillen_2,
author="Grayson, D.",
editor="Stein, M.R.",
title="{Higher algebraic K-theory: II}",
booktitle="Algebraic K-Theory",
year="1976",
publisher="Springer Berlin Heidelberg",
address="Berlin, Heidelberg",
pages="217--240",
isbn="978-3-540-37964-5"
}

@article{farrelly2020review,
  title={A review of quantum cellular automata},
  author={Farrelly, T.},
  journal={Quantum},
  volume={4},
  pages={368},
  year={2020},
  publisher={Verein zur F{\"o}rderung des Open Access Publizierens in den Quantenwissenschaften}
}

@article{Yang_2026,
   title={{Categorifying Clifford QCA}},
   volume={407},
   ISSN={1432-0916},
   url={http://dx.doi.org/10.1007/s00220-026-05596-3},
   DOI={10.1007/s00220-026-05596-3},
   number={4},
   journal={Communications in Mathematical Physics},
   publisher={Springer Science and Business Media LLC},
   author={Yang, B.},
   year={2026},
   month=Mar }

@article{WAGONER1972349,
title = {Delooping classifying spaces in algebraic K-theory},
journal = {Topology},
volume = {11},
number = {4},
pages = {349-370},
year = {1972},
issn = {0040-9383},
doi = {https://doi.org/10.1016/0040-9383(72)90031-6},
url = {https://www.sciencedirect.com/science/article/pii/0040938372900316},
author = {J.B. Wagoner}
}

@article{951cf383-5a4d-3775-b8b2-0bafd2f162fe,
 ISSN = {00029947},
 URL = {http://www.jstor.org/stable/1998675},
 author = {C. A. Weibel},
 journal = {Transactions of the American Mathematical Society},
 number = {2},
 pages = {621--635},
 publisher = {American Mathematical Society},
 title = {{$KV$-Theory of Categories}},
 volume = {267},
 year = {1981}
}

@article{GalatiusRandalWilliams2017,
  author  = {Galatius, S. and Randal-Williams, O.},
  title   = {{Homological Stability for Moduli Spaces of High Dimensional Manifolds. II}},
  journal = {Annals of Mathematics},
  series   = {Second Series},
  volume   = {186},
  number   = {1},
  year     = {2017},
  pages    = {127--204},
  doi      = {10.4007/annals.2017.186.1.4}
}

@article{Mikkola2010,
author={Mikkola, K.
and Sasane, A.},
title={{Bass and Topological Stable Ranks of Complex and Real Algebras of Measures, Functions and Sequences}},
journal={Complex Analysis and Operator Theory},
year={2010},
month={May},
day={01},
volume={4},
number={2},
pages={401-448},
issn={1661-8262},
doi={10.1007/s11785-009-0009-1},
url={https://doi.org/10.1007/s11785-009-0009-1}
}

@article{Kallen1980,
author = {Kallen, W. van der},
journal = {Inventiones mathematicae},
pages = {269-295},
title = {Homology Stability for Linear Groups.},
url = {http://eudml.org/doc/142749},
volume = {60},
year = {1980},
}

@article{Vasershtein1971,
author={Vasershtein, L. N.},
title={Stable rank of rings and dimensionality of topological spaces},
journal={Functional Analysis and Its Applications},
year={1971},
month={Apr},
day={01},
volume={5},
number={2},
pages={102-110},
issn={1573-8485},
doi={10.1007/BF01076414},
url={https://doi.org/10.1007/BF01076414}
}

@misc{Hofstadter1976Gplot,
  author       = {Hofstadter, Douglas R.},
  title        = {{File:Gplot by Hofstadter.jpg}},
  year         = {1976},
  howpublished = {Wikimedia Commons},
  note         = {Early rendering of Hofstadter's butterfly, \url{{https://commons.wikimedia.org/wiki/File:Gplot_by_Hofstadter.jpg}}. Accessed June 2026}
}

@article{Sah1986,
author = {Sah, Chih-Han},
journal = {Commentarii mathematici Helvetici},
pages = {308-348},
title = {{Homology of classical Lie groups made discrete, I. Stability theorems and Schur mulipliers.}},
url = {http://eudml.org/doc/140055},
volume = {61},
year = {1986},
}

@Inbook{Spanier1966,
author="Spanier, E. H.",
title="Algebraic Topology",
year="1966",
publisher="Springer New York",
address="New York, NY",
isbn="978-1-4684-9322-1",
doi="10.1007/978-1-4684-9322-1"
}

@book{may2011more,
  title={More Concise Algebraic Topology: Localization, Completion, and Model Categories},
  author={May, J.P. and Ponto, K.},
  isbn={9780226511795},
  series={Chicago Lectures in Mathematics},
  url={https://books.google.com/books?id=QCuvnoqsMEgC},
  year={2011},
  publisher={University of Chicago Press}
}

\end{document}